\documentclass{elsarticle}

\usepackage{fullpage}
\usepackage{amsmath}
\usepackage{amsthm}
\usepackage{amsfonts}
\usepackage{amssymb}
\usepackage{enumerate}
\usepackage{mathrsfs}
\usepackage{mathtools}
\usepackage{tikz}

\usetikzlibrary{arrows}
\usetikzlibrary{decorations.pathreplacing}

\def\theenumi{\arabic{enumi}}

\def\theenumii{\alph{enumii}}
\def\p@enumii{\theenumi.}

\def\theenumiii{\arabic{enumiii}}
\def\p@enumiii{(\theenumi)(\theenumii)}

\def\p@enumiv{\p@enumiii.\theenumiii}

\def\A{\aleph}

\def\z{\bar{z}}

\newtheorem{construction}{Construction}[section]
\newtheorem{theorem}[construction]{Theorem}

\newtheorem{corollary} [construction]{Corollary}

\newtheorem{definition} [construction]{Definition}

\newtheorem{lemma} [construction]{Lemma}

\newtheorem{remark} {Remark}
\newcommand{\Mod}[1]{\ (\mathrm{mod}\ #1)}

\begin{document}
\begin{frontmatter}

\title{The number of $s$-separated $k$-sets in various circles}
\address[UNSL]{Departamento de Matem\'atica, Universidad Nacional del San Luis, San Luis, Argentina}

\author[UNSL]{Emiliano J.J. Estrugo}\ead{juan.estrugo.tag@gmail.com} 
\author[UNSL]{Adri\'{a}n Pastine}\ead{adrian.pastine.tag@gmail.com}

\begin{abstract} This article studies the number of ways of selecting $k$ objects arranged in 
$p$ circles of sizes $n_1,\ldots,n_p$ such that no two selected ones have less than $s$
objects between them. If $n_i\geq sk+1$ for all $1\leq i \leq p$, this number is shown to be
$\frac{n_1+\ldots+n_p}{k}\binom{n_1+\ldots+n_p-sk-1}{k-1}$. A combinatorial proof of this claim
is provided, and some nice combinatorial formulas are derived.
\end{abstract}

\end{frontmatter}

\section{Introduction}
The study of separated sets in circles and lines has been of great interest in combinatorics throughout the years.
The first important result was in 1943, \cite{Kap}, where Kaplansky proved that the number of ways to choose $k$ 
elements arranged in a circle (resp. a line) of $n$ elements such that no two are consecutive is 
$\frac{n}{n-k}\binom{n-k}{k}$ (resp. $\binom{n-k+1}{k}$).
Later Konvalina \cite{Kon} showed that the number of ways of selecting $k$ elements arranged in a circle of $n$ elements, with no two selected elements having unit separation, i.e. having exactly
one object between them, is $\binom{n-k}{k}+\binom{n-k-1}{k-1}$ if $n\geq 2k+1$.


Afterwards Mansour and Sun \cite{Man} found that the number of ways of selecting $k$ out of $n$ elements arrayed in a 
circle such that no two selected ones are separated by $m-1, 2m-1, \dots , pm-1$ elements is given by $\frac{n}{n-pk}
\binom{n-pk}{k}$, where $m, p \geq 1$ and $n \geq mpk + 1$. New proofs of this result were given in \cite{Chen} 
and \cite{guo1}. In particular \cite{Chen} provides a combinatorial proof, using partitions of $\mathbb{Z}_n$.
The technique using partitions was further explored in \cite{guo2}. 

Recently, in \cite{eldesouky} the problem was generalize to study sets where no two elements are at even distance,
or no two elements are at odd distance.

On the other hand, in 2003, Talbot \cite{Tal} studied $k$ element sets 
in a circle of size $n$, with the additional condition that no two elements can be
at distance lesser than or equal to $s$. The main result in \cite{Tal} is an Erd\"{o}s-Ko-Rado type theorem,
which characterizes the families maximum size of such sets, satisfying that no two 
are disjoint. Nevertheless,
we are interested one of the lemmas from said work, which states that the number these sets
containing a fixed element is $\binom{n-ks-1}{k-1}$.

Our aim is to study the number $k$ element sets in several circles of sizes $n_1,\ldots,n_p$.
In Section \ref{seccion2circulos} we will study the number of $k$ element sets having a fixed element
in two circles of sizes $n_1$ and $n_2$. In Section \ref{seccionvarioscirculoscon1} we generalize
this result to several circles of sizes $n_1,\ldots,n_p$. Finally, in Section \ref{seccionprincipal}
we use the results obtained with a fixed element to count the number of $k$ element sets in several
circles of sizes $n_1,\ldots,n_p$.

\section{$s$-separated $k$-sets in two circles with a fixed element}\label{seccion2circulos}
We begin this section by introducing notation. The set of the first $n$ positive integers is denoted by
$[n]\coloneqq\{1,2,\dots,n\}$.
In order to distinguish between the elements of different circles, we introduce the following.
\begin{definition}
	Let $n_1,n_2,\dots n_p$ be positive integers and denote the set 
	\[
  [n_1,n_2,\dots,n_p] =	\left([n_1] \times \{1\}\right) \cup 	\left([n_2] \times \{2\}\right) \cup \cdots \left([n_p] \times \{p\}\right).    
	\]
\end{definition}

Thus the elements of the $ith$ circle are those whose second coordinate is $i$.

\begin{definition}
	Given a positive integer $s$ we say that a $k$ element subset of $[n_1,\dots,n_p]$ is \textit{$s$-separated} 
	if it does not contain two elements in a circle with less than
	$s$ elements between them. 
\end{definition}

When we are consider the elements of an $s$-separated set, we sometimes say that the elements are $s$-separated.
 

Using the notation we just introduced, the lemma from \cite{Tal} mentioned in 
the introduction can be written as follows.

\begin{lemma}\label{teorematalbot}\cite{Tal}
If $\mathcal{A}_1^{s,k}(n) = \{ A \in [n]_k^s:\: 1\in A \}$ then
\[ \left| \mathcal{A}_1^{s,k}(n) \right| = \binom{n-ks-1}{k-1}. \]
\end{lemma}

On the other hand, the following is a weaker version of the main result from \cite{Man}.
\begin{theorem}\label{teoremamansour}\cite{Man}
\[|[n]_k^s|=\frac{n}{n-sk}\binom{n-sk}{k}\]
\end{theorem}

In this section we establish a bijection between the set $\mathcal{A}_{(1,1)}^{s,k}(n_1,n_2) = 
\{A\in [n_1,n_2]_k^s: \: (1,1)\in A \}$ and the set 
in $\mathcal{A}_1^{s,k}(n_1+n_2) = \{ A \in [n_1+n_2]_k^s:\: 1\in A \}$. Together with Lemma \ref{teorematalbot}
this yields the equality $\mathcal{A}_{(1,1)}^{s,k}(n_1,n_2) = \binom{n_1+n_2-ks-1}{k-1}$. 
 
First, we define a function from $[n_1,n_2]$ to $[n_1+n_2]$ and its inverse. These functions will not be the needed 
bijection as some $s$-separated sets will turn into sets that are not $s$-separated. To obtain the bijection
we will introduce a process that will move the problematic elements (the ones that are not $s$-separated)
in order to obtain $s$-separated sets.

\begin{definition}
Let $f:[n_1,n_2]\rightarrow [n_1+n_2]$ be defined by

\[ 
f(i,j) = \begin{cases}
i & \text{if}\ j=1\\
n_1+i & \text{if}\ j=2.
\end{cases}
\]	
\end{definition}

So $f$ transforms the the two circles $[n_1]\times\{1\},[n_2]\times\{2\}$ into one circle $[n_1+n_2]$ 
by relabeling the elements and making $(1,2)$ come right after $(n_1,1)$ in the cyclic order. 

Next we introduce the inverse of $f$ defined by $g:[n_1+n_2] \rightarrow [n_1,n_2]$

\begin{definition}
Let $g:[n_1+n_2]\rightarrow [n_1,n_2]$ be defined by
	\[ 
	g(i) = \begin{cases}
	(i,1) & \text{if}\ 1 \leq i \leq n_1\\
	(i,2) & \text{if}\ n_1+1 \leq i \leq n_1+n_2.
	\end{cases}
	\]	
\end{definition}

The image of some $s$-separated sets (in their respective circles) under the functions $f$ and $g$ will not be $s$-
separated sets although those sets will remain $k$-sets. Notice that if $A\in\mathcal{A}_{(1,1)}^{s,k}(n_1,n_2)$, then 
$f(A)\in \mathcal{A}_1^{s,k}(n_1+n_2)$ if and only if $A\cap (\{n_2-s+1,\ldots,n_2\}\times\{2\})=\emptyset$.
If $f(A)\not\in \mathcal{A}_1^{s,k}$, then $A\cap (\{n_2-s+1,\ldots,n_2\}\times\{2\})$ consists of a single
element $(n_2-d,2)$. But as $A$ is $s$-separated and $(1,1)\in A$, we have $(n_1-d,1)\not\in A$. 
Then we can switch the element $(n_2-d,2)$ by the element $(n_1-d,1)$, which through $f$ will yield
a set in $\mathcal{A}_1^{s,k}(n_1+n_2)$ unless $A\cap (\{n_1-d-s+1,\ldots,n_1-d\}\times\{1\})=(n_1-d-d',1) \neq \emptyset$. 
But as $A$ is $s$-separated and $(n_2-d,2)\in A$, $(n_2-d-d',2)\not\in A$. Thus
we can switch again and keep going. Notice that in order to do these switches 
we are identifying the elements in the circles
in descending order.
\begin{align*}
(n_1,1)\quad &\longleftrightarrow \; (n_2,2)\\
(n_1-1,1)\quad &\longleftrightarrow \; (n_2-1,2)\\
&\quad \vdots\\
(n_1-s,1)\quad &\longleftrightarrow \; (n_2-s,2)\\
&\quad \vdots\\
(n_1-s(k-1),1)\quad &\longleftrightarrow \; (n_2-s(k-1),2). 
\end{align*}

We are ready now to introduce the processes that yield the bijection.
In the following definitions the first coordinate of an ordered pair is of importance,
thus we use the projection notation: $\pi_1(a,b)=a$.

\begin{definition}[zig]\label{definicionzeta}
 Let $A \in \mathcal{A}_{(1,1)}^{s,k}(n_1,n_2)$,  $n_1 \geq sk+1$, $n_2 \geq sk$, and
 
 \[
   Z_0 = A,\quad a_{-1} = n_1 + 1,\quad z_{-1} = n_2 + 1.
 \]

  For $i=0,1,\dots,k-2$ let
 
 \[
 B_i = \{ z_{i-1} - 1, \dots , z_{i-1}-s \} \times \{ i\Mod{2}\}.
 \]
 
 If $Z_i \cap B_i = \emptyset$ then the \textit{zig} of $A$ is $\mathcal{Z}(A)=Z_i$, 
 the \textit{z-order} of $A$ is $i$ and the process stops; else

 \begin{align*}
 a_i &= \pi_1 (Z_i \cap B_i),\\
 d_i &= z_{i-1}-a_i,\\
 z_i &= a_{i-1}-d_i.
 \end{align*}
 
 \[
 Z_{i+1} = (Z_i - \{ (a_i, i \Mod{2}) \}) \cup \{ (z_i,i+1\Mod{2} \}.
 \]
\end{definition}
 
\begin{remark}\label{Remark1}
Note that $1 \leq d_i \leq s$ and $f(Z_{j-1})\setminus\left(Z_j\cap B_j\right)$ is $s$-separated in $[n_1+n_2]$.
\end{remark}

\begin{definition}[zag]\label{definicionzetatecho}
Let $\aleph \in \mathcal{A}_1^{s,k}(n_1+n_2)$, $n_1 \geq sk+1$, $n_2 \geq sk$, and

\[
\widehat{Z}_0 = g(\aleph),\quad \bar{a}_{-1} = n_2 + 1,\quad \bar{z}_{-1} = n_1 + 1.
\]

For $i=0,1,\dots,k-2$ let

\[
C_i = \{ \bar{z}_{i-1} - 1, \dots , \bar{z}_{i-1}-s \} \times \{ (i+1)\Mod{2}\}.
\]

If $\widehat{Z}_i \cap C_i = \emptyset$ then the \textit{zag} of $\aleph$ is 
$\mathcal{\widehat{Z}}(\aleph)=\widehat{Z}_i$ the \textit{$\widehat{z}$-order} of $\A$ is $i$ and the process stops; else

 \begin{align*}
	\bar{a}_i &= \pi_1 (\widehat{Z}_i \cap C_i),\\
	\bar{d}_i &= \bar{z}_{i-1}-\bar{a}_i,\\
	\bar{z}_i &= \bar{a}_{i-1}-\bar{d}_i.
	\end{align*}

\[
\widehat{Z}_{i+1} = (\widehat{Z}_i - \{ (\bar{a}_i, i+1\Mod{2} \}) \cup \{ (\bar{z}_i, i\Mod{2} \}.
\]
\end{definition}
\begin{remark}\label{Remark2}
Again, note that $1 \leq \bar{d}_i \leq s$ and $\widehat{Z}_{j-1} \setminus \left( \widehat{Z}_j\cap C_j \right)$ is $s$-separated in $[n_1,n_2]$.
\end{remark}

We want $f(\mathcal{Z}(A))$ (resp. $\widehat{\mathcal{Z}}(g(\A))$) to be $s$-separated $k$-sets. 
By Remark \ref{Remark1} if $a_i\geq 1$, then $a_i\leq z_{i-1}$.  This means that the element that
prevents $f(Z_i)$ from being $s$-separated keeps getting smaller. 
Notice that at each step of Definition \ref{definicionzeta}
$|Z_i\cap B_i|\leq 1$. Even more, $\sum_{j=0}^i d_j\leq s(i+1)$, and $a_i=z_{-1}-\sum_{j=0}^id_j$. 
Hence $a_i\leq n_1+1 - s(i+1)$ if $i\equiv 1 \Mod{2}$ or $a_i\leq n_2+1-s(i+1)$
if $i\equiv 0 \Mod{2}$.

Assume the $z$-order of $A$. Then $f(Z_{k-2})$,
and $k-1$ elements in total had to be moved around.
If $k$ is even, then $(z_{k-2},1)\in Z_{k-1}$, and if $k$ is odd then $(z_{k-2},2)$. But 

\begin{align*}
z_{k-2}&\geq n_1+1-s(k-1)\geq s+1&\text{if $k$ is even}\\
z_{k-2}&\geq n_2+1-s(k-1)\geq s &\text{if $k$ is odd}.\\
\end{align*}
Then if $k$ is even, $(1,1)\in Z_{k-2}$ is $s$-separated from $(z_{k-2},1)$. 
On the other hand if $k$ is odd
$(1,2)\not\in A$, as
$A$ is $s$-separated and the $z$-order of $A$ is $k-2$ ($a_0=n_2+1-d_0\in A$). Hence if $k$ is odd, $(z_{k-2},2)$ is 
$s$-separated from the other elements of $Z_{k-2}$ and $(1,1)\in Z_{k-2}$. 
The only elements that are not separated in $Z_{k-2}$ are $(1,1)$ with
$(z_0,1)$. Thus $f(Z_{k-2})$ is $s$-separated, and $(1,1)\in f(Z_{k-2})$.
 
A similar argument shows the same for $\widehat{\mathcal{Z}}(\aleph)$. This has been summarized in the following lemma.

\begin{lemma}\label{lema1fijo}
	Let $A\in \mathcal{A}_{(1,1)}^{s,k}(n_1,n_2)$ and $\aleph\in \mathcal{A}_1^{s,k}(n_1+n_2)$.
	Then $f(\mathcal{Z}(A))$ is $s$-separated in $[n_1+n_2]$, $\widehat{\mathcal{Z}}(\aleph)$ is 
	$s$-separated in $[n_1,n_2]$ and $(1,1)\in \mathcal{Z}(A) \cap \widehat{\mathcal{Z}}(\aleph)$.
\end{lemma}	

\begin{proof}
The result follows from the discussion preceding the lemma.
\end{proof}

Working from Definition \ref{definicionzeta}, $(a_i,i\Mod{2})\in A$ for every $i$
and $(z_i,i+1\Mod{2})=(a_{i-1}-d_i,i+1\Mod{2})\not\in A$ because $A$ is $s$-separated and $d_i\leq s$.
Notice that $Z_i$ is obtained from $Z_{i-1}$ by deleting $(a_i,i\Mod{2})$ 
and adding $(z_i,i+1\Mod{2})=(a_{i-1}-d_i,i+1\Mod{2})$. So at each step an element the belonged to $A$
is taken out, and an element that was not in $A$ is added. Hence the number of elements remains unchanged 
and $|Z_i|=|A|=k$.
This works in a similar fashion for $\widehat{Z}_i$. We then obtain the following.

\begin{lemma}\label{lemaksets}
	Let $A\in \mathcal{A}_{(1,1)}^{s,k}$ and $\aleph \in \mathcal{A}_1^{s,k}$. The sets $Z_i(A)$ and 
	$\widehat{Z}(\A)$ have $k$-elements for every $i$. 
\end{lemma}

\begin{proof}
The proof follows from the discussion preceding the lemma.
\end{proof}

We are ready to provide the bijections.
\begin{definition}\label{definicionFgrande}
	Define $F:\mathcal{A}_{(1,1)}^{s,k}(n_1,n_2)\rightarrow [n_1+n_2]$ by
	\[F(A)=f(\mathcal{Z}(A)).\]
\end{definition}

\begin{definition}\label{definicionGgrande}
	Define $G:\mathcal{A}_1^{s,k}(n_1+n_2)\rightarrow [n_1,n_2]$ by
	\[G(\aleph)=\mathcal{\widehat{Z}}(\aleph).\]
\end{definition}

The remaining of this section is dedicated to prove that $F$ is a bijection from
$\mathcal{A}_{(1,1)}^{s,k}(n_1,n_2)$ onto $\mathcal{A}_1^{s,k}(n_1+n_2)$ (and that $G$ is its inverse).


\begin{lemma}\label{lemaksetsseparados}
	Let $A\in \mathcal{A}_{(1,1)}^{s,k}(n_1,n_2)$ and $\aleph \in \mathcal{A}_1^{s,k}(n_1+n_2)$ then 
	$F(A)\in \mathcal{A}_1^{s,k}(n_1+n_2)$ and $G(\aleph) \in \mathcal{A}_{(1,1)}^{s,k}(n_1,n_2)$.  
\end{lemma}

\begin{proof}
By Lemma \ref{lema1fijo}, $1\in F(A)$ and $(1,1)\in G(\aleph)$, and they both are $s$-separated. By Lemma \ref{lemaksets}
both $F(A)$ and $G(\aleph)$ contain $k$ elements. Therefore $F(A)\in \mathcal{A}_1^{s,k}(n_1+n_2)$ and $G(\aleph) \in \mathcal{A}_{(1,1)}^{s,k}(n_1,n_2)$.  
	\end{proof}
	
Thus $F:\mathcal{A}_{(1,1)}^{s,k}(n_1,n_2)\rightarrow \mathcal{A}_1^{s,k}(n_1+n_2)$ and 
$G:\mathcal{A}_1^{s,k}(n_1+n_2)\rightarrow \mathcal{A}_{(1,1)}^{s,k}(n_1,n_2)$.
We can now prove that they are inverses of one another, i.e. $G=F^{-1}$.
	
	\begin{lemma}\label{biyec}
	If $F$ and $G$ are defined as in Definitions \ref{definicionFgrande} and \ref{definicionGgrande},
	then $G=F^{-1}$.
	\end{lemma}
	
	\begin{proof}
	We will prove that $G(F(A))=A$, and $F(G(\A))=\A$ for every $A\in\mathcal{A}_{(1,1)}^{s,k}(n_1,n_2)$
	and $\A \in\mathcal{A}_{1}^{s,k}(n_1,+n_2)$.

Let $A\in \mathcal{A}_{(1,1)}^{s,k}(n_1,n_2)$ and let $\A = F(A)$. 
Then $\bar{a}_0=z_0$ and the first switch will be 
$\z_0=a_0$ which means that $\bar{a}_i=z_i$, $\bar{d}_i=d_i$ and $\bar{z}_i=a_i$. 
Hence the $z$-order of $A$ is the same as the $\widehat{z}$-order of $\A$, thus we have that

\[ G(\aleph) = G(F(A)) = A. \]

Let now $\A\in\mathcal{A}_{1}^{s,k}(n_1,+n_2)$ and $A=G(\A)$. Applying the same reasoning as before, we have that 

\[ F(A) = F(G(\A)) = \A, \]

so $G\circ F = id$ and $F\circ G = id$.
	\end{proof}
	
Thus we have that $F$ is a bijection from $\mathcal{A}_{(1,1)}^{s,k}(n_1,n_2)$ onto $\mathcal{A}_{1}^{s,k}(n_1+n_2)$.
Which means that $|\mathcal{A}_{(1,1)}^{s,k}(n_1,n_2)|=|\mathcal{A}_{1}^{s,k}(n_1+n_2)|=\binom{n_1+n_2-sk-1}{k-1}$.

\begin{theorem}\label{21circ}
	 Let $n_1\geq sk+1$ and $n_2\geq sk$. Then 
	\[
	\left| \mathcal{A}_{(1,1)}^{s,k}(n_1,n_2)\right|=\binom{n_1+n_2-sk-1}{k-1}.
	\]
\end{theorem} 

\begin{proof}
	From Lemma \ref{biyec}, $\left|\mathcal{A}_{(1,1)}^{s,k}(n_1,n_2)\right|=\left|\mathcal{A}_1^{s,k}(n_1+n_2)\right|$. Thus 
	Lemma \ref{teorematalbot} yields
	
		\[
		\left| \mathcal{A}_{(1,1)}^{s,k}(n_1,n_2)\right|=\binom{n_1+n_2-sk-1}{k-1}.
		\]
\end{proof}

 \section{$s$-separated $k$-sets in several circles with a fixed element}\label{seccionvarioscirculoscon1}
 
 In this section, we show that the number of $s$-separated $k$-sets in $p$ circles of sizes $n_1,\dots,n_p$  having a fixed element, is equal to the number of $s$-separated $k$-sets in one circle of size $n_1+\cdots+n_p$ having a fixed element.
 
By $\mathcal{A}_{(1,1)}^{s,k}(n_1,\ldots,n_p)$ we denote all the $s$-separated $k$-sets in $[n_1,\ldots,n_p]$ 
cointaining the element $(1,1)$.
Notice that any set in $\mathcal{A}_{(1,1)}^{s,k}(n_1,\ldots,n_p)$ with $j$ elements in the first $p-1$ circles
can be seen as a set in 
$\mathcal{A}_{(1,1)}^{s,j}(n_1,\ldots,n_{p-1})\cup \left([n_p]^s_{k-j}\times\{p\}\right)$. Thus, adding over 
all possible $j$ we obtain.
\begin{equation}\label{eq1}
|\mathcal{A}_{(1,1)}^{s,k}(n_1,\ldots,n_p)|=\sum_{j=1}^{k}|\mathcal{A}_{(1,1)}^{s,j}(n_1,\ldots,n_{p-1})|\times
|[n_p]^s_{k-j}|.
\end{equation}

We can use Equation \ref{eq1} inductively to obtain the following.

 \begin{theorem}\label{teorema1fijovarioscirc}
 	Let $n_1\geq sk+1$ and $n_i\geq sk$ for $i=2,\dots,p$. Let  
 	
 	\[
 	\mathcal{A}_{(1,1)}^{s,k}(n_1,\dots,n_p) = \{ A\in [n_1,\dots,n_p]_k^s: \: (1,1)\in A \}
 	\] 
 	
 	then 
 	
 	\[ \left| \mathcal{A}_{(1,1)}^{s,k}(n_1,\dots,n_p) \right| = \binom{N-sk-1}{k-1}, \]
 	
 where $N = \sum_{i=1}^{p}n_i$.
 \end{theorem}

\begin{proof}
Notice that Lemma \ref{teorematalbot} states

\[
  |\mathcal{A}_{(1,1)}^{s,k}(n_1+\cdots+n_p)| = \binom{n_1+\ldots+n_p-sk-1}{k-1}.
  \]
  We will show by induction on $p$, the number of circles, that 
\[
 	 |\mathcal{A}_{(1,1)}^{s,k}(n_1,\dots,n_{p})| = |\mathcal{A}_{(1,1)}^{s,k}(n_1+\dots+n_{p})|, \]    
 	 and so the result will follow from Lemma \ref{teorematalbot}.
  The case $p = 2$ is Theorem \ref{21circ}.\\
 	Assume that
 	
 	 \[
 	 |\mathcal{A}_{(1,1)}^{s,k}(n_1,\dots,n_{p-1})| = |\mathcal{A}_{(1,1)}^{s,k}(n_1+\dots+n_{p-1})|. \]  
 	
 	Then, Equation \ref{eq1} gives 
 	
 \begin{equation}\label{eq2}
 |\mathcal{A}_{(1,1)}^{s,k}(n_1,\ldots,n_p)|=\sum_{j=1}^{k}|\mathcal{A}_{(1,1)}^{s,j}(n_1+\cdots+n_{p-1})|\times
 |[n_p]^s_{k-j}|.
 \end{equation}

On the other hand, equation \ref{eq1} applied to a circle of size $n_1+\ldots+n_{p-1}$ and a circle of size $n_p$
implies

 \begin{equation}\label{eq3}
 |\mathcal{A}_{(1,1)}^{s,k}(n_1+\ldots+n_{p-1},n_p)|=\sum_{j=1}^{k}|\mathcal{A}_{(1,1)}^{s,j}(n_1+\cdots+n_{p-1})|\times
 |[n_p]^s_{k-j}|.
 \end{equation}
 
 Then the left sides of Equation \ref{eq2} equals the left side of Equation \ref{eq3},
 
 \[
 |\mathcal{A}_{(1,1)}^{s,k}(n_1,\ldots,n_p)| =  |\mathcal{A}_{(1,1)}^{s,k}(n_1+\ldots+n_{p-1},n_p)|.
 \]
 
 By Theorem \ref{21circ},
  
 \[
 |\mathcal{A}_{(1,1)}^{s,k}(n_1+\ldots+n_{p-1},n_p)| = \binom{N-sk-1}{k-1}
 \]

  and by Lemma \ref{teorematalbot},
  
  \[
  |\mathcal{A}_{(1,1)}^{s,k}(n_1+\cdots+n_p)| = \binom{N-sk-1}{k-1}.
  \]
  
 Therefore 
 
 \[
 \left| \mathcal{A}_{(1,1)}^{s,k}(n_1,\dots,n_p) \right| = \binom{N-sk-1}{k-1}=
 |\mathcal{A}_{(1,1)}^{s,k}(n_1+\cdots+n_p)|
 \]
and the proof by induction follows.  
\end{proof}

Theorem \ref{teorema1fijovarioscirc} counts the number of sets in $[n_1,\ldots,n_p]^s_k$ containing $(1,1)$.
Notice that if $n_j\geq sk+1$, then the elements can be relabeled to count the number of sets containing
any element $(a,j)$ in the $j$-th circle.
 
 \begin{corollary}\label{corolariocualquierfijo}
  	Let $n_j\geq sk+1$ and for $i=1,\dots,p$, $n_i\geq sk$. Let  $(a,j)\in [n_j]\times\{j\}$ and
 	
 	\[
 	\mathcal{A}_{(a,j)}^{s,k}(n_1,\dots,n_p) = \{ A\in [n_1,\dots,n_p]_k^s: \: (a,j)\in A \}
 	\] 
 	
 	then 
 	
 	\[ \left| \mathcal{A}_{(a,j)}^{s,k}(n_1,\dots,n_p) \right| = \binom{N-sk-1}{k-1}, \]
 	
 where $N = \sum_{i=1}^{p}n_i$.
 \end{corollary}
 \begin{proof}
 The corollary follows from Theorem \ref{teorema1fijovarioscirc} after relabeling the elements and circles
 such that $(a,j)$ is labeled $(1,1)$.
 \end{proof}

Notice that the equalities 
\begin{align*}
|\mathcal{A}_{(1,1)}^{s,k}(n_1,n_2)|&=\binom{n_1+n_2-sk-1}{k-1}\\
|\mathcal{A}_{(1,1)}^{s,k}(n_1)|&=\binom{n_1-sk-1}{k-1}\\
&\text{and}\\
|[n_2]^s_k|&=\frac{n}{n-sk}\binom{n-sk}{k},
\end{align*} 
together with Theorem \ref{teorema1fijovarioscirc} yield the following.
  \begin{corollary}\label{corocombin1}
 If $m\geq sk+1$ and $n\geq sk$ then
 	\[ \sum_{j=0}^{k-1}\binom{n-s(k-j)-1}{k-j-1} \frac{m}{m-sj}\binom{m-sj}{j-1} = \binom{m+n-sk-1}{k-1}. \]
 \end{corollary}
  \section{$s$-separated $k$-sets in several circles}\label{seccionprincipal}
In this section we will use Corollary \ref{corolariocualquierfijo} to count the number of 
$s$-separated $k$-sets in several circles of various sizes, provided all of the sizes
are at least $sk+1$.  
  In order to do so, we add over every element of the circle the number of $s$-separated sets
  containing it, and then divide by the number of elements in each set.
 \begin{theorem}\label{principal}
 	Let $n_i\geq sk+1$ for $i=1,\dots,p$. Then 
 	
 	\[ |[n_1,\dots,n_p]_k^s| = \frac{N}{k}\binom{N-sk-1}{k-1} \] 
 where $N = \sum_{i=1}^{p}n_i$.	
 \end{theorem}
\begin{proof} 
As $n_i\geq sk+1$, Theorem \ref{teorema1fijovarioscirc} ensures that the number of $s$-separated $k$-sets containing a 
fixed elements from the $i$-th circle is $\binom{N-sk-1}{k-1}$. Then, adding over all elements these numbers yields
\[
N\mathcal{A}^{s,k}_{(1,1)}(n_1,\dots,n_p)=N\binom{N-sk-1}{k-1}.\]
Notice that in doing this we counted each $s$-separated $k$-set $k$ times, once per element in said set.
Thus
\[
N\mathcal{A}^{s,k}_{(1,1)}(n_1,\dots,n_p) = k |[n_1,\dots,n_p]_k^s|,\]
and applying Corollary \ref{corolariocualquierfijo}
\[
 |[n_1,\dots,n_p]_k^s| = \frac{N}{k}\binom{N-sk-1}{k-1}. \] 
 \end{proof}

Using that any $s$-separated
$k$-set with $j$ elements in one circle  and $k-j$ elements in another is the union of 
an $s$-separated $j$-set in one circle with an $s$-separated $k-j$-set in the other we get the
following:

\[
|[n_1,n_2]^s_k|=\sum_{j=0}^k |[n_1]^s_j|\times |[n_2]^s_{k-j}|.
\]
We can now follow the ideas from Corollary \ref{corocombin1} and apply Theorem \ref{principal} 
to get
\begin{corollary}
\[
\frac{n_1+n_2}{k}\binom{n_1+n_2-sk-1}{k-1}=\sum_{j=0}^k\frac{n_1}{n_1-sj}\binom{n_1-sj}{j}
\frac{n_2}{n_2-s(k-j)}\binom{n_2-s(k-j)}{k-j}\]
\end{corollary}
 
\section{Acknowledgements}
This work was partially supported by the Universidad Nacional de San Luis , grant PROIPRO 03-2216, 
and MATH AmSud, grant 18-MATH-01. 
The first author was supported by a doctoral scholarship from 
Consejo Nacional de Investigaciones Cient\'{i}ficas y 
T\'{e}cnicas (CONICET).


\end{document}